\documentclass[12pt]{article}

\usepackage{amsmath,amsthm,amsfonts,latexsym,amsopn,verbatim,amscd,amssymb,}
\usepackage{amssymb}
\usepackage{amsfonts}
\usepackage{hyperref}
\usepackage{graphics,color}
\usepackage{graphicx}

\theoremstyle{plain}
\newtheorem{Thm}{Theorem}[section]
\newtheorem{Lem}[Thm]{Lemma}
\newtheorem{Prop}[Thm]{Proposition}
\newtheorem{Cor}[Thm]{Corollary}
\theoremstyle{definition}

\newtheorem{Def}[Thm]{Definition}

\numberwithin{equation}{section}

\newcommand{\bnum}{\begin{enumerate}}
\newcommand{\enum}{\end{enumerate}}
\begin{document}
\begin{center}
\textbf{Weakly $I$-clean rings}\\
\end{center}

\begin{center}
 Ajay Sharma and Dhiren Kumar Basnet\\
\small{\it Department of Mathematical Sciences, Tezpur University,
 \\ Napaam, Tezpur-784028, Assam, India.\\
Email: ajay123@tezu.ernet.in \\
 dbasnet@tezu.ernet.in}
\end{center}
\noindent \textit{\small{\textbf{Abstract:}  }} In this article, we introduce the concept of weakly $I$-clean ring, for any ideal $I$ of a ring $R$. We show that, for  an ideal $I$ of a ring $R$, $R$ is uniquely weakly $I$-clean if and only if $R/I$ is semi boolean and idempotents can be lifted uniquely weakly modulo $I$ if and only if for each $a\in R$, there exists a central idempotent $e\in R$ such that either $a-e\in I$ or $a+e\in I$ and $I$ is idempotent free. As a corollary, we characterize weakly $J$-clean ring. Also we study various properties of weakly $I$-clean ring.

\bigskip

\smallskip

\noindent \small{\textbf{\textit{$2010$ Mathematics Subject Classification:}}  16N40, 16U99.} \\
\smallskip

\bigskip

\section{INTRODUCTION}
$\mbox{\hspace{.5cm}}$

In this article, all rings are associative ring with unity unless otherwise indicated. Here Jacobson radical of a ring $R$ is denoted by $J(R)$. Set of all units, set of all nilpotent elements and set of all idempotent elements are respectively denoted by $U(R)$, $Nil(R)$ and $Idem(R)$. A ring $R$ is called abelian if every idempotent commutes with every element of $R$. An ideal $I$ of a ring $R$ is said to be idempotent free if $e^2=e\in I$, then $e=0$. For an ideal $I$ of a ring, we say that idempotents can be lifted (respectively, lifted uniquely, lifted centrally) modulo $I$, if for any $x\in R$ with $x-x^2 \in I$, there exists $e\in Idem(R)$ (respectively, unique $e\in Idem(R)$, central $e\in Idem(R)$) such that $x-e\in I$. A ring $R$ is said to be exchange ring if for any $x\in R$ there exists $e\in Idem(R)$ such that $e\in aR$ and $1-e\in (1-a)R$. In $1977$, W. K. Nicholson \cite{nicholson1977lifting} introduced the concept of clean rings as a subclass of exchange rings. He defined a ring $R$ to be clean ring if every element of $R$ can be written as a sum of a unit and an idempotent. Again as a subclass of clean rings A. J. Diesl \cite{disl2013thss} introduced the notion of nil clean ring in the year $2013$. A ring $R$ is said to be a nil clean ring if for any $x\in R$, $x=n+e$, for some $e\in Idem(R)$ and $n\in Nil(R)$. In $2010$, H. Chen \cite{JCHEN} introduced the notion of $J$-clean ring as a ring $R$, where every element of $R$ can be written as a sum of an idempotent and an element from the Jacobson radical. V. A. Hiremath and H. Sharad \cite{hir2013shar} generalised the concepts of nil clean ring and $J$-clean ring to $I$-clean ring in the year $2013$. For an ideal $I$ of a ring $R$, $R$ is said to be $I$-clean ring if for any $x\in R$, there exists $e\in Idem(R)$ such that $x-e\in I$ and if $e$ is unique in the expression, then the ring $R$ is said to be uniquely $I$-clean ring. \par
 Here we introduce the notion of weakly $I$-clean ring, for any ideal $I$ of a ring $R$. For an ideal $I$ of a ring $R$, we say $R$ is weakly $I$-clean ring if for any $x\in R$ there exists $e\in Idem(R)$ such that $x-e\in I$ or $x+e\in I$ and if $e$ is unique, then $R$ is said to be uniquely weakly $I$-clean ring. We study various properties of weakly $I$-clean ring and uniquely weakly $I$-clean ring.

\section{Weakly $I$-clean ring}
\begin{Def}
  For an ideal $I$ of $R$, we say that $R$ is weakly $I$-clean ring, if for each $x\in R$ there exists $e\in Idem(R)$ such that $x-e\in I$ or $x+e\in I$. Also $R$ is said to be uniquely weakly $I$-clean ring if for any $x\in R$, there exists a unique $e\in Idem(R)$ such that $x-e\in I$ or $x+e\in I$.
\end{Def}
It is easily seen that every $I$-clean ring is weakly $I$-clean. The converse is not true because $\mathbb{Z}$ is weakly $3\mathbb{Z}$-clean but not $3\mathbb{Z}$-clean. Note that if $R$ is weakly $I_1$-clean and not weakly $I_2$-clean then $R$ is not necessarily weakly $I_1\cap I_2$-clean. For a commutative ring $R$, $R$ is weakly nil clean ring if $I=Nil(R)$.
\begin{Thm}
  Let $\{R_{\alpha}\}$ be a collection of rings and $I_{\alpha}\leqslant R_{\alpha}$. Then $R=\prod R_{\alpha}$ is weakly $I=\prod I_{\alpha}$-clean \textit{if and only if} each $R_{\alpha}$ is weakly $I_{\alpha}$-clean and at most one $R_{\alpha}$ is not $I_{\alpha}$-clean.
\end{Thm}
\begin{proof}
  If $R$ is weakly $I$ clean, then each $R_{\alpha}$ is weakly $I_{\alpha}$-clean. Suppose that $R_{\alpha_1}$ and $R_{\alpha_2}$ ($\alpha_1\neq \alpha_2$) are not $I_{\alpha_1}$-clean and $I_{\alpha_2}$-clean respectively. There exists $x_{\alpha_1}\in R_{\alpha_1}$ such that $x_{\alpha_1}\neq w_1-e_1$, where $w_1\in I_{\alpha_1}$ and $e_1\in Idem(R_{\alpha_1})$. Similarly there exists $x_{\alpha_2}\in R_{\alpha_2}$ such that $x_{\alpha_2}\neq w_2+e_2$, where $w_2\in I_{\alpha_2}$ and $e_2\in Idem(R_{\alpha_2})$. Define $x=(x_{\alpha})\in R$ by $x_{\alpha}=x_{\alpha}$ for $\alpha \in \{\alpha_1, \alpha_2\}$, otherwise $x_{\alpha}=0$. Then $x\neq w+e$ or $x\neq w-e$, for any $w\in I$ and $e\in Idem(R)$.\\
  Conversely, let each $R_{\alpha}$ be $I_{\alpha}$-clean, then clearly $R$ is $I$-clean. Assume that $R_{\alpha_0}$ is weakly $I_{\alpha_0}$-clean but not $I_{\alpha_0}$-clean and other $R_{\alpha}$'s are $I_{\alpha}$-clean. Let $x=(x_{\alpha})\in R$, so in $R_{\alpha_0}$, $x_{\alpha_0}= w_{\alpha_0}+e_{\alpha_0}$ or $x_{\alpha_0}= w_{\alpha_0}-e_{\alpha_0}$, where $w_{\alpha_0}\in I_{\alpha_0}$ and $e_{\alpha_0}\in Idem(R_{\alpha_0})$. If $x_{\alpha_0}= w_{\alpha_0}+e_{\alpha_0}$, then assume $x_{\alpha}= w_{\alpha}+e_{\alpha}$ $(\alpha \neq \alpha_0)$ and if $x_{\alpha_0}= w_{\alpha_0}-e_{\alpha_0}$, then we assume $x_{\alpha}= w_{\alpha}+e_{\alpha}$ $(\alpha \neq \alpha_0)$, where $w_{\alpha}\in I_{\alpha}$ and $e_{\alpha}\in Idem(R_{\alpha})$. Set $w=(w_{\alpha})$ and $e=(e_{\alpha})$, then either $x=w+e$ or $x=w-e$.
\end{proof}
\begin{Lem}\label{L1}
  If $R$ is weakly $I$-clean ring, then $J(R)\subseteq I$.
\end{Lem}
\begin{proof}
  Let $x\in J(R)$, then there exists $e\in Idem(R)$ such that either $x-e\in I$ or $x+e\in I$. If $x-e\in I$, then $x-e=w$, so $(x-w)^2=x-w$, which implies $x(1-x)\in I$. But $1-x$ is unit, hence $x\in I$.
\end{proof}
The converse of Lemma \ref{L1} is not true as for the ideal $I=5\mathbb{Z}$ in $\mathbb{Z}$, $J(\mathbb{Z})=\{0\}\subseteq 5\mathbb{Z}$, but $\mathbb{Z}$ is not weakly $5\mathbb{Z}$-clean. Note that, if $R$ is weakly $I$-clean, then $R$ is weakly $B$-clean for any ideal $B$ of $R$ with $I\subseteq B$.
\begin{Def}
  Let $I$ be an ideal of $R$. We say that idempotents can be lifted weakly modulo $I$ if for $x^2-x\in I$, there exists $e\in Idem(R)$ such that either $x-e\in I$ or $x+e\in I$. Also we say that idempotents can be lifted uniquely weakly modulo $I$ if for $x^2-x\in I$, there exists a unique $e\in Idem(R)$ such that either $x-e\in I$ or $x+e\in I$.
\end{Def}
\begin{Thm}\label{T2}
  If $I$ is an ideal of $R$ such that $R/I$ is boolean and idempotents lift weakly modulo $I$, then $R$ is weakly $I$-clean.
\end{Thm}
\begin{proof}
  Let $x\in R$. We have $x^2-x\in I$, so there exists $e\in Idem(R)$ such that, either $x-e\in I$ or $x+e\in I$. Hence $R$ is weakly $I$-clean.
\end{proof}
Clearly $\mathbb{Z}$ is weakly $3\mathbb{Z}$-clean but $\mathbb{Z}_3=\mathbb{Z}/3\mathbb{Z}$ is not boolean, so the converse of Theorem \ref{T2} is not true.
\begin{Lem}\label{L2}
  Let $R$ be a ring and $e=f+n$, where $e, f\in Idem(R)$, $n\in Nil(R)$ and $f\in C(R)$, then $n=0$.
\end{Lem}
\begin{proof}
  See Lemma 2.5\cite{hir2013shar}.
\end{proof}
For any ring $R$, let $T_n(R)$ be the ring of all upper triangular matrices over $R$ with usual addition and multiplication.
Using Lemma \ref{L2} we can say that if idempotents of a ring $R$ are central, then the idempotents of the ring $S=\{[a_{ij}]\in T_n(R)\,|\, a_{ii}=a_{jj}$ for all $i,j\}$ is the set $Idem(S)=\{eI_n\,|\,e\in Idem(R)\}$.
\begin{Thm}
  If $R$ is weakly $I$-clean ring and idempotents in $R$ are central, then $S=\{[a_{ij}]\in T_n(R)\,|\, a_{ii}=a_{jj}$ for all $i,j\}$ is weakly $I'$-clean, where $I'=\{[a_{ij}]\in S\,|\, a_{ii}\in I \}$.
\end{Thm}
\begin{proof}
  Clearly $I'$ is an ideal of $S$. Since idempotents are central, so $Idem(S)=\{eI_n\,|\,e\in Idem(R)\}$. Let $A=[a_{ij}]\in S$ and $a_{ii}=a$. Since $R$ is weakly $I$-clean, so there exist $e\in Idem(R)$ and $w\in I$ such that $a=w+e$ or $a=w-e$. If $a=w+e$, then set $A=B+eI_n$ and $B=(b_{ij})$, where $b_{ij}=w$ if $i=j$, otherwise $b_{ij}=a_{ij}$. Hence $B\in I'$ and $eI_n\in Idem(S)$. Similarly if $a=w-e$, then we can prove that $A=B-eI_n$, where $B\in I'$ and $eI_n\in Idem(S)$.
\end{proof}
An element $x$ of a ring $R$ is said to be quasi-regular if $1-x\in U(R)$.
\begin{Prop}\label{P1}
  If $R$ is a uniquely weakly $I$-clean ring for an ideal $I$ of $R$, then the following hold:
  \begin{enumerate}
    \item $I$ is idempotent free.
    \item $I$ contains all the quasi-regular elements.
    \item $R$ is abelian.
  \end{enumerate}
\end{Prop}
\begin{proof}
  \begin{enumerate}
    \item Let $e^2=e\in I$. Since $1=1+0=(1-e)+e$, hence $e=0$.
    \item Let $a\in R$ be quasi-regular. Since $R$ is uniquely weakly $I$-clean ring, so $1-a=(1-e)+y$ or $1-a=-(1-e)+y$, for some $e\in Idem(R)$ and $y\in I$. Now $(1-a)e=ye\in I$, but $1-a\in U(R)$ which implies $e\in I$. So by $(i)$, $e=0$ and hence $a\in I$.
    \item Let $e\in Idem(R)$. By $(ii)$ $Nil(R)\subseteq I$. Let $x=e+er-ere$, clearly $x\in Idem(R)$. Now $x=(e+er-ere)+0=e+(er-ere)$, Hence by uniqueness $er-ere=0$ $i.e.$ $er=ere$. Similarly we can show that $re=ere$.
  \end{enumerate}
\end{proof}
\begin{Lem}\label{L3}
  Let $I$ be an ideal of a ring $R$ with $N(R)\subseteq I$. Then the following are equivalent:
  \begin{enumerate}
    \item Idempotents can be lifted uniquely weakly modulo $I$.
    \item Idempotents can be lifted weakly modulo $I$, $R$ is abelian and $I$ is idempotent free.
    \item Idempotents can be lifted centrally weakly modulo $I$ and $I$ is idempotent free.
  \end{enumerate}
 \end{Lem}
 \begin{proof}
   (i) $\Rightarrow$ (ii) proof is same as Lemma 2.8(1) \cite{hir2013shar}.\\
 (ii) $\Rightarrow$ (iii) is trivial.\\
 (iii) $\Rightarrow$ (i)\\
  Let $\overline{x}$ be an idempotent in $R/I$. By assumption, there exists a central idempotent $e\in R$ such that either $x-e\in I$ or $x+e\in I$. Suppose that $f\in Idem(R)$ for which either $x-f\in I$ or $x+f\in I$, then\\
 \textbf{Case I:}\\
  If $x-e\in I$ and $x-f\in I$ or $x+e\in I$ and $x+f\in I$, then $e-f\in I$. Similar to the proof of Lemma $2.8$ \cite{hir2013shar} we can show that $e=f$.\\
 \textbf{Case II:}\\
  If $x-e\in I$ and $x+f\in I$ or $x+e\in I$ and $x-f\in I$, then $e+f\in I$. As $I$ is an ideal so $(e+f)(1-e)\in I$, implies $f(1-e)\in I$. But $e$ is central idempotent, so $f(1-e)\in Idem(R)$, $i.e.$ $f=fe$, as $I$ is idempotent free. Similarly we can show that $e=ef$. Hence $e=f$.\\
 From both the cases we conclude that idempotents can be lifted uniquely weakly modulo $I$.

 \end{proof}
 \begin{Def}
   A ring $R$ is called semi boolean ring if either $x^2=x$ or $x^2=-x$, for all $x\in R$.
 \end{Def}
 Clearly boolean rings are semi boolean rings but converse is not true as $\mathbb{Z}_3$ is semi boolean ring but not boolean.
 \begin{Thm}\label{T4}
   For an ideal $I$ of a ring $R$ the following are equivalent:
   \begin{enumerate}
     \item $R$ is uniquely weakly $I$-clean ring.
     \item $R/I$ is semi boolean and idempotents can be lifted uniquely weakly modulo $I$.
   \end{enumerate}
 \end{Thm}
 \begin{proof}
   (i)$\Rightarrow$ (ii)\\
    Let $\overline{x}\in R/I$. Then there exists unique $e\in Idem(R)$ such that either $x=i+e$ or $x=i-e$, where $i\in I$. If $x=i+e$, then $\overline{x}^2=\overline{x}$, also if $x=i-e$, then $\overline{x}^2=-\overline{x}$, hence $R/I$ is semi boolean. Next part is obvious.\\
   (ii)$\Rightarrow$ (i)\\
    Let $x\in R$, by $(ii)$ either $\overline{x}^2=\overline{x}$ or $\overline{x}^2=-\overline{x}$. If $\overline{x}^2=\overline{x}$, then by assumption there exists unique $e\in Idem(R)$ such that either $x-e\in I$ or $x+e\in I$. Also if $\overline{x}^2=-\overline{x}$, then $(-\overline{x})^2=-\overline{x}$, again by assumption there exists unique $e\in Idem(R)$ such that either $-x+e\in I$ or $-x-e\in I$.
   \end{proof}
   A ring $R$ is said to be weakly $J$-clean ring if for any $x\in R$, $x=j+e$ or $x=j-e$, where $j\in J(R)$ and $e\in Idem(R)$. Also $R$ is said to be uniquely weakly $J$-clean ring if for any $x\in R$, there exists a unique $e\in Idem(R)$ such that $x-e\in J(R)$ or $x+e\in J(R)$.

   \begin{Cor}
    For a ring $R$, the following are equivalent.
     \begin{enumerate}
       \item $R$ is uniquely weakly $J$-clean.
       \item $R/J(R)$ is semi boolean and idempotents can be lifted uniquely weakly modulo $J(R)$.
     \end{enumerate}
   \end{Cor}
   A ring $R$ is said to be weakly nil clean ring if for any $x\in R$, $x=e+n$ or $x=-e+n$, where $n\in Nil(R)$ and $e\in Idem(R)$. Also $R$ is said to be uniquely weakly nil clean ring if for any $x\in R$, there exists a unique $e\in Idem(R)$ such that $x-e\in Nil(R)$ or $x+e\in Nil(R)$.

   \begin{Cor}
    For a commutative ring $R$, the following are equivalent.
     \begin{enumerate}
       \item $R$ is uniquely weakly nil clean.
       \item $R/Nil(R)$ is semi boolean and idempotents can be lifted uniquely weakly modulo $Nil(R)$.
     \end{enumerate}
   \end{Cor}
   \begin{Cor}\label{Cor1}
     Let $I$ be an ideal of a ring $R$. Then the following are equivalent.
     \begin{enumerate}
       \item $R$ is uniquely weakly $I$-clean.
       \item $R/I$ is semi boolean and idempotents can be lifted uniquely weakly modulo $I$.
       \item $R/I$ is semi boolean and idempotents can be lifted weakly modulo $I$, $R$ is abelian and $I$ is idempotent free.
       \item For each $a\in R$ there exists a central idempotent $e\in R$ such that either $a-e\in I$ or $a+e\in I$ and $I$ is idempotent free.
     \end{enumerate}
   \end{Cor}
   \begin{proof}
     If $R/I$ is semi boolean, then for $n\in N(R)$, $n^k=0$, for some $k\in \mathbb{N}$. Either $\overline{n}^k=\overline{n}$ or $\overline{n}^k=-\overline{n}$ but $\overline{n}^k=I$, implies $n\in I$. Hence by Lemma \ref{L3} the result follows.
   \end{proof}
   \begin{Prop}
     For an ideal $I$ of a ring $R$ the following are equivalent.
     \begin{enumerate}
       \item $I$ is prime and $R$ is uniquely weakly $I$-clean.
       \item $R$ is uniquely weakly $I$-clean and $0,1$ are the only idempotents in $R$.
       \item $R/I\cong \mathbb{Z}_2$ or $\mathbb{Z}_3$ and $I$ is idempotent free.
       \item $I$ is maximal and $R$ is uniquely weakly $I$-clean.
     \end{enumerate}
   \end{Prop}
 \begin{proof}
   (i)$\Rightarrow$ (ii)\\
    Let $e\in Idem(R)$. By Proposition \ref{P1}, $R$ is abelian, so $eR(1-e)=0\subseteq I$. Therefore $e\in I$ or $1-e\in I$, but again by Proposition \ref{P1}, $I$ is idempotent free and hence $e=0$ or $e=1$.\\
   (ii)$\Rightarrow$ (iii) \\
   From Theorem \ref{T4}, $R/I$ is semi boolean. Let $\overline{x}\in R/I$, then either $\overline{x}$ is boolean or $-\overline{x}$ is boolean. If $\overline{x}$ is boolean, then $x\in I$ or $x-1\in I$ and if $-\overline{x}$ is boolean, then $-x\in I$ or $-x-1\in I$. From both the cases, either $\overline{x}=\overline{0}$ or $\overline{x}=\overline{1}$ or $\overline{x}=\overline{-1}$. Hence $R/I\cong \mathbb{Z}_2$ or $\mathbb{Z}_3$.\\
   (iii)$\Rightarrow$ (iv) \\
   Assume $R/I\cong \mathbb{Z}_2$ or $\mathbb{Z}_3$, then clearly $I$ is a maximal ideal and $R/I$ is semi boolean.\\
   (iv)$\Rightarrow$ (i) It is obvious.
 \end{proof}
\begin{Cor}
  The following are equivalent for a non trivial ring $R$.
  \begin{enumerate}
    \item $R$ is local and uniquely weakly $J$-clean.
    \item $R$ is uniquely weakly $J$-clean and $0,1$ are the only idempotents in $R$.
    \item $R/J(R)\cong \mathbb{Z}_2$ or $\mathbb{Z}_3$
  \end{enumerate}
\end{Cor}
For a ring $R$, let $V$ be an $R-R$ bimodule, which is a ring not necessarily with $1$. Let $I(R;V)$, the ideal extension of $R$ by $V$, is defined to be the additive abelian group $I(R;V)=R\oplus V$, where the multiplication is defined by $(r,v)(s,w)=(rs, rw+vs+vw)$, for all $v,w\in V$ and $r,s\in R$.
\begin{Lem}\label{L4}
  Let $R$ be a ring, $V$ be an $R$-$R$-bimodule which is also a ring (not necessarily with unity) and let $S=I(R;V)$ be the ideal extension. If $V$ is idempotent free and for each $e\in Idem(R)$, $ev=ve$ for all $v\in V$, then $Idem(S)=\{(e,0)\,\,|\,\,e\in Idem(R)\}$.
\end{Lem}
\begin{proof}
  See Lemma $2.22$ \cite{hir2013shar}.
\end{proof}
\begin{Lem}
  Let $R$ be a ring, $V$ be an $R$-$R$-bimodule which is also a ring (not necessarily with $1$) and let $S=I(R;V)$ be the ideal extension. Then the following are equivalent.
  \begin{enumerate}
    \item For every $v\in V$, there exists $w\in V$ such that $v+w+vw=0$.
    \item $(1,v)\in U(S)$ for all $v\in V$.
    \item $U(S)=\{(u,v)\in S\,\,|\,\, u\in U(R)\}$.
    \item $J(S)=J(R)\times V$.
  \end{enumerate}
  Further if any of four equivalent conditions holds, then $J(R)=\{r\in R\,\,|\,\, (r,0)\in J(S)\}$.
\end{Lem}
\begin{proof}
  See Lemma $2.23$ \cite{hir2013shar}.
\end{proof}
\begin{Prop}
  Let $R$ be a ring and let $V$ be an $R$-$R$-bimodule which is also an idempotent-free ring not necessarily with $1$. Let $S=I(R;V)$ be the ideal extension of $R$ by $V$. Then the following are equivalent.
  \begin{enumerate}
    \item $S=I(R;V)$ is uniquely weakly $I'$-clean for some ideal $I'$ of $S$.
    \item \begin{enumerate}
            \item $R$ is uniquely weakly $I$-clean for some ideal $I$ of $R$.
            \item If $e\in Idem(R)$, then $ev=ve$ for all $v\in V$.
          \end{enumerate}
  \end{enumerate}
\end{Prop}
\begin{proof}
(i)$\Rightarrow$ (ii) \\
(b) Let $e\in Idem(R)$, then clearly $(e,0)\in Idem(S)$. Since $S$ is weakly $I'$-clean, so by Proposition \ref{P1}, $S$ is abelian. For $v\in V$, $(e,0)(v,0)=(v,0)(e,0)$ and hence $ev=ve$, as required.\\
(a) Consider $I=\{x\in R\,|\,(x,0)\in I'\}$, then $I$ be an ideal of $R$. Since $S$ is abelian, so by Lemma \ref{L4}, $Idem(S)=\{(e,0)\,| \, e\in Idem(R)\}$.\\
Claim: $R$ is uniquely weakly $I$-clean ring.\\
Let $x\in R$. As $S$ is uniquely weakly $I'$-clean, so by Corollary \ref{Cor1}, there exists a central idempotent $(e,0)\in S$ such that $(x,0)=(e,0)+(y,0)$ or $(x,0)=-(e,0)+(y,0)$, for some $(y,0)\in I'$. Therefore $x-e=y\in I$ or $x+e=y\in I$, so $R$ is weakly $I$-clean ring. Let $e^2=e\in I$, which implies $(e,0)\in I'$. Since $I'$ is idempotent free, so $e=0$ and hence $I$ is idempotent free. By Corollary \ref{Cor1}, $R$ is uniquely weakly $I$-clean ring.\\
(ii) $\Rightarrow$ (i)\\
Consider $I'=I\times V$, then $I'$ is an ideal of $S$. Let $(r,v)\in S$. Since $R$ is uniquely weakly $I$-clean ring, so there exists a central idempotent $e\in R$ such that $r-e\in I$ or $r+e\in I$. Now $(r,v)=(e,0)+(r-e,v)$ or $(r,v)=-(e,0)+(r+e,v)$ implies $(r,v)-(e,0)\in I'$ or $(r,v)+(e,0)\in I'$ and hence $S$ is weakly $I'$-clean ring. By (b) and \ref{L4}, we have $Idem(S)=\{(e,0)\in Idem(R)\}$. For $(e,0)^2=(e,0)\in I'$, $e\in I$, but by Corollary \ref{Cor1}, $I$ is idempotent free, so $(e,0)=(0,0)$. Hence $S$ is uniquely weakly $I'$-clean ring.
\end{proof}

\section{Acknowledgement}
The first Author was supported by Government of India under DST(Department of Science and Technology), DST-INSPIRE registration no IF160671.


\begin{thebibliography}{99}
\bibitem{JCHEN} Chen, H. On strongly J-clean rings, \textit{Comm. Algebra}, $38(10):$ $3790-3804$, $2010$.
\bibitem{disl2013thss}Diesl, A. J. Nil clean rings, \textit{J. Algebra} $383$: $197-211$, $2013$.
\bibitem{hir2013shar}Hiremath, V. A. and Sharad, H. Using ideals to provide a unified approach to uniquely clean rings. \textit{J. Aust. Math. Soc.}, $96(2)$: $258-274$, $2013$.
\bibitem{nicholson1977lifting}Nicholson, W. K., Lifting idempotents and exchange rings, \textit{Trans. Amer. Math. Soc.}, $229$: $269-278$, $1977$.


\end{thebibliography}
\end{document}